\documentclass[reqno]{amsart}
\usepackage{amsmath}
\usepackage{amssymb}
\usepackage{amsthm}
\usepackage{cases}
\usepackage{multicol}
\usepackage{algorithm}
\usepackage{algpseudocode}
\usepackage{mathrsfs}
\usepackage{hyperref}

\title{BMR freeness for icosahedral family}

\author{Shunsuke Tsuchioka}
\address{Graduate School of Mathematical Sciences, University of Tokyo,
Komaba, Meguro, Tokyo, 153-8914, Japan}
\thanks{The first author was supported in part by JSPS Kakenhi Grants 26800005.}
\email{tshun@kurims.kyoto-u.ac.jp}

\date{Oct 11, 2017}
\keywords{complex reflection groups,
Coxeter grops, Weyl groups,
Iwahori-Hecke algebras,
Ariki-Koike algebras,
cyclotomic Hecke algebras,
BMR freeness conjecture,
rational Cherednik algebras,
symplectic reflection algebras,
hyperplane arrangements,
monodromy,
noncommutative Gr\"obener basis,
Bergman's diamond lemma,
Newman's lemma,
Knuth-Bendix completion,
Church-Rosser property,
rewriting systems,
critical pairs,
Shephard-Todd classification}
\subjclass[2010]{Primary~20C08, Secondary~68W30}

\newtheorem{Thm}{Theorem}[section]
\newtheorem{Def}[Thm]{Definition}

\newtheorem{Prop}[Thm]{Proposition}

\newtheorem{Rem}[Thm]{Remark}
\newtheorem{Cor}[Thm]{Corollary}
\newtheorem{Ex}[Thm]{Example}

\usepackage{graphicx}
\usepackage{tikz}

\usetikzlibrary{arrows}
\tikzstyle{every picture}+=[remember picture]
\tikzstyle{na} = [baseline=-.5ex]
\tikzstyle{mine}= [arrows={angle 90}-{angle 90},thick]

\def\Llleftarrow{%
\lower2pt\hbox{\begingroup
\tikz
\draw[shorten >=0pt,shorten <=0pt] (0,3pt) -- ++(-1em,0) (0,1pt) -- ++(-1em-1pt,0) (0,-1pt) -- ++(-1em-1pt,0) (0,-3pt) -- ++(-1em,0) (-1em+1pt,5pt) to[out=-105,in=45] (-1em-2pt,0) to[out=-45,in=105] (-1em+1pt,-5pt);
\endgroup}
}

\DeclareMathOperator*{\restprod}%
{\mathchoice{\ooalign{\ensuremath{\displaystyle\prod}\crcr\ensuremath{\displaystyle\coprod}}}%
  {\ooalign{\ensuremath{\textstyle\prod}\crcr\ensuremath{\textstyle\coprod}}}%
  {\ooalign{\ensuremath{\scriptstyle\prod}\crcr\ensuremath{\scriptstyle\coprod}}}%
  {\ooalign{\ensuremath{\scriptscriptstyle\prod}\crcr\ensuremath{\scriptscriptstyle\coprod}}}%
}

\makeatletter
\renewcommand{\boxed}[2][\fboxsep]{{%
  \setlength{\fboxsep}{#1}\fbox{\m@th$\displaystyle#2$}}}
\makeatother

\newcommand{\MA}{\mathbb{A}}

\newcommand{\EXS}[1]{Y_j}

\DeclareMathOperator{\IRR}{Irr}

\DeclareMathOperator{\RED}{\mathsf{norm}}

\DeclareMathOperator{\PAIR}{\mathsf{pair}}
\DeclareMathOperator{\AVO}{\mathsf{Irr}}
\DeclareMathOperator{\AVOID}{\mathsf{avoid}}
\DeclareMathOperator{\LT}{\mathsf{LT}}
\DeclareMathOperator{\LEX}{\mathsf{lex}}
\DeclareMathOperator{\RLEX}{\mathsf{rlex}}

\newcommand{\mygeq}{\succeq}

\newcommand{\myge}{\succ}
\newcommand{\NOTDIV}{\nmid}
\newcommand{\cor}{k}
\newcommand{\MON}[1]{{#1}^{\ast}}
\newcommand{\FP}[1]{\cor[#1^{\ast}]}
\newcommand{\FPZ}[1]{\mathbb{Z}[#1^{\ast}]}

\begin{document}
\maketitle

\begin{abstract}
We verify the Brou\'e-Malle-Rouquier (BMR) freeness for cyclotomic Hecke algebras associated with complex reflection groups $G_{17}$, $G_{18}$, $G_{19}$
in the Shephard-Todd classification. Together with results of Ariki, Ariki-Koike, Brou\'e-Malle, Marin, Marin-Pfeiffer and Chavli, this settled affirmatively the BMR freeness 
conjecture. 
Our verification is inspired by Bergman's diamond lemma and requires 1GB memory and 5 days calculation on a PC.
\end{abstract}

\section{Introduction}
The Coxeter group $W$ is a major player in Lie theory.
The Coxeter presentation of $W$ gives that of the Iwahori-Hecke algebra $H$
and that of the braid group $B$.
$H$ is a quotient of the group algebra of $B$ and free of rank $|W|$.

In their seminal paper~\cite{BMR}, Brou\'e-Malle-Rouquier (BMR, for short)
generalized the constructions to complex reflection groups (CRG, for short).
In this paper, we call their Hecke algebras cyclotomic Hecke algebras 
according to a tradition in modular representation theory of the symmetric groups.
Because the Coxeter groups and real reflection groups are the same,
it is natural to expect that they share similar properties to those for Coxeter groups.
Among the list of expectations, the so-called BMR freeness conjecture
states that the cyclotomic Hecke algebra associated with a CRG $W$ is free of rank $|W|$ over a base ring.

According to Shephard-Todd, CRG are classified into an infinite series $G(de,e,r)$ and 
exceptional groups indexed as $G_{4},\cdots,G_{37}$ (among them $G_{23},G_{28},G_{30},G_{35},G_{36},G_{37}$ are 
the Coxeter groups of type $H_3,F_4,H_4,E_6,E_7,E_8$ respectively).
As of the beginning of 2017, BMR freeness conjecture is known to be true other than $G_{17},G_{18},G_{19},G_{20},G_{21}$
and Marin gave a proof for $G_{20},G_{21}$ in February~\cite[Theorem 2.2]{Ma2}.
For $G_{21}$, his proof is a computer proof.
The goal of this paper is to give a computer proof for the remaining CRG $G_{17},G_{18},G_{19}$.
Though we employ computer, our method is in a sense uniform
and applicable to rank 2 CRG $G_{4},\cdots,G_{22}$.


\begin{Thm}\label{maintheorem}
For $4\leq n\leq 22$, let $H_n$ be an associative algebra (with 1) over $\mathbb{Z}[a_1,\cdots,a_{\ell}]$ as in Table \ref{tab1},\ref{tab2},\ref{tab3}.
As a $\mathbb{Z}[a_1,\cdots,a_{\ell}]$-module, $H_n$ is free of rank $|G_n|$.
\end{Thm}

Here the definitions of $H_n$ (see ~\cite[Definition 2.1]{Ma2}) are different from the 
original~\cite[Definition 4.21]{BMR}
but equivalent~\cite[Proposition 2.3.(ii)]{Ma2}.


Our method, inspired by Bergman's diamond lemma~\cite{Ber}
as noncommutative Gr\"obner basis theory
and some of Marin's result for $G_{21}$, is just computational.
On the other hand there are conceptual approaches (see ~\cite{Los,ER,Eti}) that operates well to establish various weak analog of
BMR freeness (for some CRG) that give us a hope that full BMR freeness can be proved at least withough computer.

Unfortunately but clearly, the proof given in this paper never give us such an understanding. 
Nonetheless, a reason for writing this paper is that our method seems to be applicable 
to other concretely defined ``small'' algebras and might be worth recording the detail.
Other reason is that 
it seems that the reduction system $R_n$ in Definition \ref{finalRS} 
is strongly normalizing (i.e., any choice of nontrivial applications leads to a unique normal form) though I could not prove that.
If it is the case, $R_n$ is expected to play a role in the study of cyclotomic Hecke algebras of rank 2
since it is no wonder some statements for exceptional case can only be proved through computer typically in Lie theory.

\hspace{0mm}

\noindent{\bf Acknowledgments.} The major part of the work was carried out
at Caffe Strada when the author visited University of California Berkeley 
for Berkeley-Tokyo Summer School ``Geometry, Representation Theory and Mathematical Physics'' in August, 2017.
The author thanks Toshitake Kohno for giving him an opportunity.

\hspace{0mm}

\noindent{\bf Notations and Conventions.} In the rest, $X$ is a set.
We denote by $\MON{X}$ the free monoid generated by $X$, i.e.,
the set of finite words of $X$ with concatenate operation as multiplication.
We write the empty word (the multiplicative identity) as $\varepsilon$. 
For $L\subseteq\MON{X}$, we define $\AVO(L)=\{W'\in\MON{X}\mid \forall W\in L, W\NOTDIV W'\}$
where we denote by $W|W'$ for $W,W'\in\MON{X}$ if $AWB=W'$ for some $A,B\in\MON{X}$.

In the rest, $\cor$ is a commutative ring with 1.
We define $\FP{X}$ to be the monoid ring of $\MON{X}$ over $\cor$. This is a free associative $\cor$-algebra with 1 generated by $X$, i.e., the $\cor$-algebra
of noncommutative (NC, for short) polynomials of variables $X$ with coefficients $\cor$.

\begin{table}
\begin{tabular}{clclcll}
$n$ & $G_n$              & $|G_n|$ & $H_n$           & $\ell$ & $X_n$         & {order on $X_n^{\ast}$} \\ \hline
4   & $s^3=t^3=1$        & 24      & $s^3=a_1s^2+a_2s+1$ & 2  & {$\{s,t\}$}   & {${\RLEX_{(X_n,t>s)}}$}   \\
{}  & $sts=tst$          & {}     & $t^3=a_1t^2+a_2t+1$ & {}  & {}            & {}                     \\
{}  & (note $s\equiv t$) & {}     & $tst=sts$          & {}  & {}            & {}                     \\ \hline
5   & $s^3=t^3=1$        & 72     & $s^3=a_1s^2+a_2s+1$ & 4   & {$\{s,t\}$}   & {${\RLEX_{(X_n,t>s)}}$}   \\
{}  & $stst=tsts$        & {}     & $t^3=a_3t^2+a_4+1$  & {}  & {}            & {}                     \\
{}  & {}                 & {}     & $stst=tsts$        & {}  & {}            & {}                     \\ \hline
6   & $s^3=t^2=1$        & 48     & $s^3=a_1s^2+a_2s+1$ & 3   & {$\{s,t\}$}   & {${\RLEX_{(X_n,t>s)}}$}   \\
{}  & $ststst=tststs$    & {}     & $t^2=a_3t+1$       & {}  & {}            & {}                     \\
{}  & {}                 & {}     & $ststst=tststs$    & {}  & {}            & {}                    \\ \hline
7   & $t^2=s^3=u^3=1$     & 144    & $t^2=a_1t+1$       & 5   & {$\{s,t,u\}$} & {${\RLEX_{(X_n,u>t>s)}}$} \\
{}  & $tsu=sut=uts$      & {}     & $u^3=a_2u^2+a_3u+1$ & {}  & {}            & {}                    \\
{}  & {}                 & {}     & $s^3=a_4s^2+a_5s+1$ & {}  & {}            & {}                    \\
{}  & {}                 & {}     & $tsu=sut$          & {}  & {}            & {}                    \\
{}  & {}                 & {}     & $sut=uts$          & {}  & {}            & {}      
\end{tabular}
\caption{(modified) BMR presentations for tetrahedral family}
\label{tab1}
\end{table}

\begin{table}
\begin{tabular}{clclcll}
$n$ & $G_n$                    & $|G_n|$ & $H(G_n)$                  & $\ell$ & $X_n$        & {order on $X_n^{\ast}$} \\ \hline
8  & $s^4=t^4=1$               & 96      & $s^4=a_1s^3+a_2s^2+a_3s+1$ & 3      & {$\{s,t\}$}  & {${\RLEX_{(X_n,t>s)}}$}   \\
{} & $sts=tst$                 & {}      & $t^4=a_1t^3+a_2t^2+a_3t+1$ & {}     & {}          & {}                     \\
{} & (note $s\equiv t$)        & {}      & $tst=sts$                 & {}     & {}          & {}                     \\ \hline
9  & $s^4=t^2=1$               & 192     & $s^4=a_1s^3+a_2s^2+a_3s+1$ & 4      & {$\{s,t\}$}  & {${\RLEX_{(X_n,t>s)}}$}   \\
{} & $ststst=tststs$           & {}      & $t^2=a_4t+1$              & {}     & {}           & {}                    \\
{} & {}                        & {}      & $ststst=tststs$           & {}    & {}            & {}                    \\ \hline
10 & $s^4=t^3=1$                & 288     & $s^4=a_1s^3+a_2s^2+a_3s+1$ & 5     & {$\{s,t\}$}   & {${\RLEX_{(X_n,t>s)}}$}  \\
{} & $stst=tsts$               & {}      & $t^3=a_4t^2+a_5t+1$        & {}    & {}            & {}                    \\
{} & {}                        & {}      & $stst=tsts$               & {}    & {}            & {}                    \\ \hline
11 & $t^2=s^3=u^4=1$            & 576     & $t^2=a_1t+1$              & 6     & {$\{s,t,u\}$} & {${\RLEX_{(X_n,u>t>s)}}$} \\
{} & $tsu=sut=uts$              & {}     & $u^4=a_2u^3+a_3u^2+a_4u+1$ & {}    & {}            & {}                    \\
{} & {}                         & {}     & $s^3=a_5s^2+a_6s+1$        & {}    & {}            & {}                    \\
{} & {}                         & {}     & $tsu=sut$                 & {}    & {}            & {}                    \\
{} & {}                         & {}     & $sut=uts$                 & {}    & {}            & {}                    \\ \hline
12 & $t^2=s^2=u^2=1$             & 48     & $t^2=a_1t+1$              & 1     & {$\{s,t,u\}$} & {${\RLEX_{(X_n,u>t>s)}}$} \\
{} & $tsut=suts=utsu$           & {}     & $s^2=a_1s+1$              & {}    & {}            & {}                    \\
{} & (note $s\equiv t\equiv u$) & {}     & $u^2=a_1u+1$              & {}    & {}            & {}                    \\
{} & {}                         & {}     & $tsut=suts$              & {}     & {}            & {}                    \\
{} & {}                         & {}     & $utsu=suts$              & {}     & {}            & {}                    \\ \hline
13 & $t^2=s^2=u^2=1$             & 96     & $t^2=a_1t+1$             & 2      & {$\{s,t,u\}$} & {${\RLEX_{(X_n,u>t>s)}}$} \\
{} & $tsuts=utsut$              & {}     & $s^2=a_2s+1$             & {}     & {}            & {}                    \\
{} & $suts=utsu$                & {}     & $u^2=a_2u+1$             & {}     & {}            & {}                    \\
{} & (note $s\equiv u$)         & {}     & $utsut=tsuts$            & {}    & {}             & {}                    \\
{} & {}                         & {}     & $utsu=suts$              & {}    & {}             & {}                    \\ \hline
14 & $s^3=t^2=1$                & 144    & $s^3=a_1s^2+a_2s+1$       & 3     & {$\{s,t\}$}    & {${\RLEX_{(X_n,t>s)}}$}  \\
{} & $stststst=tstststs$        & {}     & $t^2=a_3t+1$             & {}     & {}            & {}                    \\
{} & {}                         & {}     & $stststst=tstststs$      & {}    & {}             & {}                    \\ \hline
15 & $s^2=t^2=u^3=1$             & 288   & $t^2=a_1t+1$              & 4     & {$\{s,t,u\}$}  & {${\RLEX_{(X_n,u>t>s)}}$} \\
{} & $tsu=sut$                  & {}     & $s^2=a_2s+1$             & {}     & {}            & {}                    \\
{} & $utsus=tsusu$              & {}     & $u^3=a_3u^2+a_4u+1$      & {}     & {}            & {}                     \\
{} & {}                         & {}     & $tsusu=utsus$              & {}     & {}            & {}                     \\
{} & {}                         & {}     & $tsu=sut$          & {}     & {}            & {}      
\end{tabular}
\caption{(modified) BMR presentations for octahedral family}
\label{tab2}
\end{table}

\begin{table}
\begin{tabular}{clclcll}
$n$ & $G_n$                     & $|G_n|$ & $H(G_n)$                         & $\ell$ & $X_n$        & {order on $X_n^{\ast}$} \\ \hline
16 & $s^5=t^5=1$                & 600     & $s^5=a_1s^4+a_2s^3+a_3s^2+a_4s+1$ &  4     & {$\{s,t\}$}  & {${\LEX_{(X_n,t>s)}}$}  \\
{} & $sts=tst$                  & {}      & $t^5=a_1t^4+a_2t^3+a_3t^2+a_4t+1$ & {}     & {}           & {}                    \\
{} & (note $s\equiv t$)         & {}      & $tst=sts$                       & {}     & {}            & {}                    \\ \hline
17 & $s^5=t^2=1$                & 1200    & $s^5=a_1s^4+a_2s^3+a_3s^2+a_4s+1$ &  5     & {$\{s,t\}$}   & {${\LEX_{(X_n,t>s)}}$}  \\
{} & $ststst=tststs$            & {}      & $t^2=a_5t+1$                    &  {}    & {}            & {}                    \\
{} & {}                         & {}      & $tststs=ststst$                 & {}     & {}            & {}                    \\ \hline
18 & $t^5=s^3=1$                & 1800    & $t^5=a_1t^4+a_2t^3+a_3t^2+a_4t+1$ & 6      & {$\{s,t\}$}   & {${\RLEX_{(X_n,t>s)}}$}  \\
{} & $stst=tsts$                & {}      & $s^3=a_5s^2+a_6s+1$              & {}     & {}            & {}                    \\
{} & {}                         & {}      & $stst=tsts$                     & {}     & {}            & {}                    \\ \hline
19 & $s^2=t^3=u^5=1$             & 3600   & $s^2=a_1s+1$                     & 7      & {$\{s,t,u\}$} & {${\RLEX_{(X_n,u>t>s)}}$} \\
{} & $stu=tus=ust$              & {}      & $u^5=a_2u^4+a_3u^3+a_4u^2+a_5u+1$            & {}     & {}            & {}                    \\
{} & {}                         & {}      & $t^3=a_6t^2+a_7t+1$ & {}     & {}            & {}                    \\
{} & {}                         & {}      & $stu=tus$                       & {}     & {}            & {}                    \\
{} & {}                         & {}      & $ust=tus$                       & {}     & {}            & {}                    \\ \hline
20 & $s^3=t^3=1$                & 360      & $s^3=a_1s^2+a_2s+1$             & 2      & {$\{s,t\}$}   & {${\RLEX_{(X_n,t>s)}}$}   \\
{} & $ststs=tstst$              & {}      & $t^3=a_1t^2+a_2t+1$              & {}     & {}            & {}                    \\
{} & (note $s\equiv t$)         & {}      & $tstst=ststs$                   & {}     & {}            & {}                    \\ \hline
21 & $s^3=t^2=1$                & 720      & $s^3=a_1s^2+a_2s+1$             & 3      & {$\{s,t\}$}   & {${\RLEX_{(X_n,t>s)}}$}   \\
{} & $ststststst$               & {}      & $t^2=a_3t+1$                    & {}     & {}            & {}                     \\
{} & $=tststststs$              & {}      & $ststststst=tststststs$         & {}     & {}            & {}                    \\ \hline
22 & $s^2=t^2=u^2=1$             & 240     & $t^2-a_1t-1$                    & 1      & {$\{s,t,u\}$} & {${\RLEX_{(X_n,u>t>s)}}$} \\
{} & $tsuts=sutsu$              & {}      & $s^2-a_1s-1$                    & {}     & {}            & {}                     \\
{} & $\phantom{tsuts}=utsut$    & {}      & $u^2-a_1u-1$                    & {}     & {}            & {}                     \\
{} & (note $s\equiv t\equiv u$) & {}      & $sutsu=tsuts$                   & {}    & {}            & {}                      \\
{} & {}                         & {}      & $sutsu=utsut$                   & {}    & {}            & {}
\end{tabular}
\caption{(modified) BMR presentations for icosahedral family}
\label{tab3}
\end{table}

\section{Bergman's diamond lemma}
We briefly review ~\cite[\S1]{Ber}, a dialect of NC Gr\"obner basis theory.

\subsection{Reductions}
A reduction system (RS, for short) for $\FP{X}$ is a set of reductions. 
When we say a finite ordered RS, it means a finite sequenced collection of reductions $S=(s_1,\cdots,s_n)$.
For a RS $S$, we put $\LT(S)=\{W\mid (W,f)\in S\}(\subseteq\MON{X})$.

\begin{Def}
\begin{enumerate}
\item A reduction for $\FP{X}$ is a pair $(W,f)\in \MON{X}\times \FP{X}$.
\item For a reduction $s=(W,f)$ and $A,B\in\MON{X}$, we define a $\cor$-linear map $r_{A,B,s}:\FP{X}\to \FP{X}$ by
$
r_{A,B,s}\Bigl(\sum_{C}a_C C\Bigr)=a_{AWB}AfB+\sum_{C\ne AWB}a_C C.
$
\end{enumerate}
\end{Def}

We say that $r_{A,B,s}$ acts non-trivially for $\sum_{C}a_C C\in \FP{X}$ when $a_{AWB}\ne0$.

\subsection{Orders on $\MON{X}$}
\begin{Def}
A total order $\geq$ on $\MON{X}$ is called
\begin{enumerate}
\item good if
\begin{itemize}
\item it is artinian, i.e., there is no countably infinite $x_1>x_2>\cdots$ in $X$
\item it respects monoid structure on $\MON{X}$, i.e., $AWB\geq A'WB'$ whenever $A\geq A',B\geq B'$
for all $A,A',B,B',W\in\MON{X}$
\end{itemize}
\item compatible with a RS $S$ if $W>C$ for $(W,f)\in S$ with $a_C\ne 0$ where $f=\sum_{C}a_C C$
\end{enumerate}
\end{Def}

\begin{Ex}
For a set $X$ with a total order $\mygeq$, the quasi-lexicographical order $\geq_{\LEX_{(X,\mygeq)}}$ (reps. the reverse quasi-lexicographical order $\geq_{\RLEX_{(X,\mygeq)}}$) 
on $\MON{X}$ is defined for words $x=x_1\cdots x_a$ and $y=y_1\cdots y_b$ as $x >_{\LEX(X,\mygeq)} y$ (resp. $x >_{\RLEX(X,\mygeq)} y$) if
\begin{enumerate}
\item $a> b$ or
\item $a=b$, $1\leq\exists i\leq a,x_i\myge y_i,1\leq\forall j<i,x_j=y_j$ (resp. $i<\forall j\leq a,x_j=y_j$)
\end{enumerate}
\end{Ex}

\subsection{A normalization}
For a subset $Y\subseteq\MON{X}$, $\cor[Y]$ means the free $\cor$-submodule of $\FP{X}$ with basis $Y$.
For a RS $S$ and $g\in \cor[\AVO(\LT(S))]$,
we have $r_{A,B,s}(g)=g$ for $A,B\in\MON{X}$ and $s\in S$, i.e., each reduction $r_{A,B,s}$ acts
on $g$ trivially.

\begin{algorithm}
\caption{Calculate $\RED(g;S)$ if halt}
\label{algred}
\begin{algorithmic}
\Require{a finite set $X$, a total order $\geq$ on $\MON{X}$ and a NC polynomial $g\in\FP{X}$}
\Require{a finite ordered RS $S=(s_1,\cdots,s_n)$ for $\FP{X}$ where $s_i=(W_i,f_i)$}
\Loop
\State{$Z:=\{C\in\MON{X}\mid a_C\ne 0,1\leq\exists i\leq n,W_i|C\}$ (where $g=\sum_{C}a_CC$)}
\If{$Z=\emptyset$}
\State{return $g$}
\Else
\State{$W:=\max Z$, $m:=\min \{1\leq i\leq n\mid W_i|W\}$}
\State{take $A,B\in\MON{X}$ with $AW_mB=W$ where $\ell(A)$ is minimum}
\State{replace $g$ with $r_{A,B,s_m}(g)$}
\EndIf
\EndLoop
\end{algorithmic}
\end{algorithm}

\begin{Def}
Let $S=(s_1,\cdots,s_n)$ be a finite ordered RS where $s_i=(W_i,f_i)$ and fix a total order $\geq$ on $\MON{X}$.
For $g=\sum_{C}a_CC\in\FP{X}$, $\RED(g;S)\in\cor[\AVO(\LT(S))]$ is the output of algorithm \ref{algred}.
Note that there is a possibility that the algorithm never visits \verb|return g| (i.e., loop forever).
In this case, we do not define $\RED(g;S)$.
\end{Def}

\begin{Rem}\label{redterm}
It is clear that
$\RED(g;S)$ is defined for all $g\in\FP{X}$ when $S$ is a finite ordered RS with a compatible good order on $\MON{X}$ 
(where $X$ is a finite set).
\end{Rem}

\subsection{Ambiguities, local confluences and strong normalizabilities}

\begin{Def}
For a RS $S$, we
\begin{enumerate}
\item associate a two-sided ideal $I_S\subseteq \FP{X}$ generated by $\{W-f\mid (W,f)\in S\}$.
\item say a 5-tuple $(\sigma,\tau,A,B,C)$ overlap ambiguity if $W_{\sigma}=AB,W_{\tau}=BC$ where $A,B,C\in\MON{X}\setminus\{\varepsilon\}$ and $\sigma=(W_{\sigma},f_{\sigma}),\tau=(W_{\tau},f_{\tau})\in S$.
\item say a 5-tuple $(\sigma,\tau,A,B,C)$ inclusion ambiguity if $W_{\sigma}=B,W_{\tau}=ABC$ where $A,B,C\in\MON{X}$ with $\sigma\ne\tau$ and $\sigma=(W_{\sigma},f_{\sigma}),\tau=(W_{\tau},f_{\tau})\in S$. 
\item define an overlap (resp. inclusion) ambiguity $(\sigma,\tau,A,B,C)$ is resolvable if 
\begin{align*}
r_{A_1,B_1,s_1}\circ\cdots\circ r_{A_u,B_u,s_u}(x)=r_{A'_1,B'_1,t_1}\circ\cdots\circ r_{A'_v,B'_v,t_v}(y)
\end{align*}
for some $A_1,\cdots,A_u,A'_1,\cdots,A'_v,B_1,\cdots,B_u,B'_1,\cdots,B'_v\in\MON{X}$ and $s_1,\cdots,s_u,t_1,\cdots,t_v\in S$
where $x=f_{\sigma}C,y=Af_{\tau}$ (resp. $x=f_{\sigma},y=Af_{\tau}C$) and $u,v\geq 0$.
\end{enumerate}
\end{Def}

Ambiguities are analogous to critical pairs in the theory of rewriting (see ~\cite[\S6.2]{BN}).
Bergman's diamond lemma is a ring theoretic analog of Newman's lemma or critical pair lemma (see ~\cite[Lemma 2.7.2,Theorem 6.2.4]{BN}).

\begin{Thm}[{\cite[Theorem 1.2]{Ber}}]\label{berthm}
Let $S$ be a RS for $\FP{X}$ having a compatible good order on $\MON{X}$.
Assume all ambiguities are resolvable.
Then, for $g\in\FP{X}$ and
for $A_1,\cdots,A_n,B_1,\cdots,B_n\in\MON{X},s_1,\cdots,s_n\in S$ such that $r_{A_is_iB_i}$ acts non-trivially on 
$r_{A_{i-1},B_{i-1},s_{i-1}}\circ\cdots\circ r_{A_1,B_1,s_1}(g)$ for all $1\leq i\leq n$ and
$h:=(r_{A_{n},B_{n},s_{n}}\circ\cdots\circ r_{A_1,B_1,s_1})(g)\in\cor[\AVO(\LT(S))]$,
$h$ only depends on $g$.
\end{Thm}

Thus, we get a well-defined map $r_S:\FP{X}\to\cor[\AVO(\LT(S))],g\mapsto (r_{A_{n},B_{n},s_{n}}\circ\cdots\circ r_{A_1,B_1,s_1})(g)$.
By ~\cite[Lemma 1.1]{Ber}, $r_S$ is $\cor$-linear.

\begin{Cor}[{\cite[Theorem 1.2]{Ber}}]
As a $\cor$-algebra, $\FP{X}/I_S$ is isomorphic to $\cor[\AVO(\LT(S))]$ with multiplication $g\cdot h=r_S(gh)$ for $g,h\in\cor[\AVO(\LT(S))]$.
\end{Cor}

\section{Proof of Theorem \ref{maintheorem}}
\subsection{Minimal reduction systems}
The following is analogous to the uniqueness of reduced Gr\"obner basis (see ~\cite[Proposition 6 in \S2.7]{CLO}) and seems to be well known.
For completeness, we state and give proofs.

\begin{Prop}\label{minS}
Let $S$ be a RS for $\FP{X}$ having a compatible good order on $\MON{X}$ whose ambiguities are resolvable.
Take $s=(W,f)\in S$ and assume $\exists t=(W',f')\in S':=S\setminus\{s\},W'|W$.
Then, all the ambiguities of $S'$ are resolvable and $I_S=I_{S'}$.
\end{Prop}

\begin{proof}
Let $(\sigma,\tau,A,B,C)$ be an overlap (resp. inclusion) ambiguity of $S\setminus\{s\}$
where $\sigma=(W_{\sigma},f_{\sigma}),\tau=(W_{\tau},f_{\tau})$
and put $x=f_{\sigma}C,y=Af_{\tau}$ (resp. $x=f_{\sigma},y=Af_{\tau}C$).
Since ambiguities of $S$ are resolvable, we have $r_S(x)=r_S(y)$ by Theorem \ref{berthm}.
Take $A_n,B_n,\in\MON{X},s_n\in S$ recursive on $n\geq 1$ so that 
$r_{A_{n},B_{n},s_{n}}$ acts non-trivially on $(r_{A_{n-1},B_{n-1},s_{n-1}}\circ\cdots\circ r_{A_1,B_1,s_1})(x)$ if possible.
In each step, we can take $s_n\ne s$ because of $W'|W$.
Again by Theorem \ref{berthm}, this choice stops for some $N\geq 0$ and $r_S(x)=(r_{A_{N},B_{N},s_{N}}\circ\cdots\circ r_{A_1,B_1,s_1})(x)$.
Apply the same argument for $y$, we see that $(\sigma,\tau,A,B,C)$ is resolvable in $S'$.
$I_S=I_{S'}$ follows from $r_S(W-f)=0$ and the same argument.
\end{proof}

\begin{Def}\label{minrs}
We say a RS $S$ is minimal if $W\NOTDIV W'$ and $W'\NOTDIV W$ for any $s=(W,f),t=(W',f')\in S$ with $s\ne t$.
\end{Def}

\begin{Prop}\label{uniqS}
Let $S$ be a RS for $\FP{X}$ whose ambiguities are resolvable with a compatible good order on $\MON{X}$ and 
let $S'$ be a RS with the same assumption.
If $I_S=I_{S'}$ and both are minimal, then $\LT(S)=\LT(S')$ (and thus $|S|=|S'|$).
\end{Prop}

\begin{proof}
Write $S=\{(W_i,f_i)\mid 1\leq i\leq s\}$ and $S'=\{(W'_i,f'_i)\mid 1\leq i\leq s'\}$.
Since $r_{S'}(W_1-f_1)=0$ by $W_1-f_1\in I_S=I_{S'}$, we have $1\leq \exists i\leq s',W'_i|W_1$. 
Similarly, $1\leq\exists j\leq s,W_j|W'_i$. By the minimalities, $j=1$ and $W_1=W'_i$.
Again, $1\leq \exists \ell\leq s',W'_{\ell}=W_2$. The minimality of $S$ implies $\ell\ne i$.
The rest is the same.
\end{proof}

The conclusion of Proposition \ref{uniqS} become ``$S=S'$ up to permutation'',
if we define the minimality of $S$ as Definition \ref{minrs} plus $\forall s=(W,f)\in S,f\in\cor[\IRR(\LT(S))]$.

\subsection{The avoiding pattern $T_n$}
\begin{Prop}\label{existsRS}
For $4\leq n\leq 22$, let $X_n$ be the set of generators of $G_n$ (identified with those of $H(G_n)$) 
and let $I_n\subseteq\mathbb{Z}[X_n^{\ast}]$ be a two-sided ideal generated by the relation of $G_n$ as in Table \ref{tab1},\ref{tab2},\ref{tab3}.
For any total order $\mygeq$ on $X_n$, choose $\geq_{\LEX_{(X_n,\mygeq)}}$ or $\geq_{\RLEX_{(X_n,\mygeq)}}$ as an order on $X_n^{\ast}$ and write it as $\geq$.
Then, there exists a finite RS $S_{n,\geq}$ for $\mathbb{Z}[X_n^{\ast}]$ compatible with $\geq$ 
such that $I_{S_{n,\geq}}=I_n$ and ambiguities are resolvable.
\end{Prop}


Existence of $S_{n,\geq}$ is carried out by computer by the procedure (1)--(4). 

\begin{enumerate}
\item let $w_1=w'_1,\cdots,w_m=w'_m$ be defining relations of $G_n$ as in Table \ref{tab1},\ref{tab2},\ref{tab3}.
\item start with a RS $S_{n,\geq}:=(\PAIR(w_1-w_1'),\cdots,\PAIR(w_m-w_m'))$.
\item if there are ambiguities in $S_{n,\geq}$ (otherwise $S_{n,\geq}$ is a desired one, so return it), pick $(\sigma,\tau,A,B,C)$ where $\sigma=(W_{\sigma},f_{\sigma}),\tau=(W_{\tau},f_{\tau})$ from them and put $z=\RED(g;S_{n,\geq})$ where $g=f_{\sigma}C-Af_{\tau}$ or $g=f_{\sigma}-Af_{\tau}C$ depending on the ambiguity is overlap or inclusion.
\item if $z\ne 0$, add $\PAIR(z)$ to $S_{n,\geq}$ and repeat (3).
\end{enumerate}

Here for $0\ne g=\sum_{C}a_CC\in \cor[X_n^{\ast}]$ with $a_Z=\pm1$ where $Z=\max\{C\in X_n^{\ast}\mid a_C\ne 0\}$,
we define a reduction $\PAIR(g)=(Z,\mp(g\mp Z))$. 
If $g\in \cor[X_n^{\ast}]$ does not satisfy the assumption, we do not define $\PAIR(g)$.
Note that in (4) $\PAIR(z)$ is defined.

\begin{Rem}\label{conce}
Though I do not know a proof,
it is no wonder by virtue of $\FPZ{X_n}/I_n\cong \mathbb{Z}G_n$ that the termination of the procedure (1)--(4) above can be proved 
for 
\begin{enumerate}
\item[(I)] not only $\geq_{\LEX_{(X_n,\mygeq)}},\geq_{\RLEX_{(X_n,\mygeq)}}$ but also any good total order on $X_n^{\ast}$ 
\item[(J)] any way of picking ambiguities in (3)
\item[(K)] any way of sequencing of relation in (1)
\item[(L)] any way of (non-trivial) applications of maps $r_{A,B,s}$ in algorithm \ref{algred}
\end{enumerate}
\end{Rem}

Unless proving Remark \ref{conce} (if true), 
existence of $S_{n,\geq}$ through (1)--(4) 
is incomplete in the sense that there is no reproducibility.
Thus, to remain logically rigorous, we must specify (J),(K) in Remark \ref{conce} for a fixed $\geq$ on $X_n^{\ast}$.
But we need not to be nervous on the detail because we will rely only on $|\AVO(\LT(S_{n,\geq}))|=|G_n|$ (see also Remark \ref{notet}).
For a given finite subset $T\subseteq X_n^\ast$ we can (in principle) calculate $|\AVO(T)|$
by algorithm \ref{algenum} when $|\AVO(T)|<\infty$ since $\AVO(T)=\AVOID(T,\varepsilon)$.

\begin{algorithm}                      
\caption{Calculate $\AVOID(T,p)=\{q\in\MON{X}\mid \exists r\in\MON{X},pr=q\in\IRR(T)\}$ for $T\subseteq\MON{X}$ if it is a finite set}
\label{algenum}
\begin{algorithmic}
\Require{a finite set $X$, a finite subset $T\subseteq\MON{X}$ and a word $p\in\MON{X}$}
\If{$\{t\in T\mid\exists u\in\MON{X},p=ut\}\ne\emptyset$}
\State{return $\emptyset$}
\Else
\ForAll{$x\in X$}
\State{call recursively $A_{x}:=\AVOID(T,px)$}
\EndFor
\State{return $\{p\}\bigsqcup\bigsqcup_{x\in X}A_{x}$}
\EndIf
\end{algorithmic}
\end{algorithm}

\begin{Rem}
There is an algorithm that determine whether $|\IRR(T)|=\infty$ or not for a finite $X$ and a finite $T\subseteq\MON{X}$
based on cycle detection in a directed graph (e.g., topological sort).
Such a refinement of algorithm \ref{algenum} is not needed for our purpose.
\end{Rem}

\begin{Def}\label{condef}
For $4\leq n\leq 22$, let $X_n$ and the order $\geq$ on $X_n^{\ast}$ as in Table \ref{tab1},\ref{tab2},\ref{tab3}.
We define $T_n=\LT(S_n)$ where $S_n$ is the RS obtained from minimizing $S_{n,\geq}$.
\end{Def}

For a finite RS $S$, minimization means to
repeat deleting $s=(W,f)$ from $S$ until there exists $t=(W',f')\in S\setminus\{s\}$ with $W'|W$.
By Proposition \ref{existsRS},
$S_{n,\geq}$ exists that implies $S_n$ exists.
By Proposition \ref{minS}, all ambiguities of $S_n$ are resolvable
and $\LT(S_n)$ is well-defined by Proposition \ref{uniqS}. 

\begin{Rem}\label{notet}
Though $T_n$ plays a crucial role in the rest, we will not present a computer code for calculation of $T_n$ 
because we do not use the fact that $T_n=\LT(S_n)$. Thus, to remain logically rigorous, it is good to think that we are just given a finite $T_n\subseteq X_n^\ast$
(say by Ramanujan) to which computer can verify $|\AVO(T_n)|=|G_n|$.
\end{Rem}


Thanks to ~\cite[Proposition 2.3.(ii)]{Ma2}, Proposition \ref{crucialprop} implies Theorem \ref{maintheorem}.

\begin{Prop}\label{crucialprop}
For $4\leq n\leq 22$, put $\MA=\mathbb{Z}[a_1,\cdots,a_{\ell}]$. 
In $H_n$, we have 
\begin{align*}
\forall w\in \IRR(T_n),\forall x\in X_n,\exists(b_{C})_{C\in\AVO(T_n)}\in\MA^{\AVO(T_n)},wx=\sum_{C\in\AVO(T_n)}b_{C}C.
\end{align*}
\end{Prop}

Recall $\ell$ is as in Table \ref{tab1},\ref{tab2},\ref{tab3}.
In the rest, we will explain how to verify Proposition \ref{crucialprop} by computer.
The origin of $T_n$ as $\LT(S_n)$ is a heuristic why Proposition \ref{crucialprop} holds.

\subsection{Proof of Proposition \ref{crucialprop}}\label{mainex}

\begin{Def}\label{mseq}
For $4\leq n\leq 22$, let $(x_1,y_1,z_1,u_1),\cdots(x_m,y_m,z_m,u_m)$ be the sequence presented in \S\ref{dat}.$n$ where $m$ is the title of \S\ref{dat}.$n$.
\end{Def}

The sequence is designed so that we can define $S^{(i)}_n=(s^{(i)}_0,s^{(i)}_1,\cdots,s^{(i)}_{N^{(i)}})$ inductively by (P),(Q),(R)
staring with the RS $S^{(0)}_n$ of $H_n$ ordered as in Table \ref{tab1},\ref{tab2},\ref{tab3}.
\begin{enumerate}
\item[(P)] For $i=1,\cdots,m$, $s^{(i-1)}_{x_i}$ and $s^{(i-1)}_{y_i}$ have an overlap (resp. inclusion) ambiguity
when $u_i=0$ (resp. $u_i=1$). More precisely, write $s^{(i-1)}_{x_i}=(W,f)$ and $s^{(i-1)}_{y_i}=(W',f')$ where $W=w_1\cdots w_p$ and $W'=w'_1\cdots w'_q$.
Then, 
\begin{itemize}
\item $w_{p-z_i+1}\cdots w_p=w'_1\cdots w'_{z_i}$ if $u_i=0$ (put $h=fw'_{z_i+1}\cdots w'_{q}-w_1\cdots w_{p-z_i}f'$)
\item $W=w'_{z_i+1}\cdots w'_{z_i+p}$ if $u_i=1$ (put $h=w'_1\cdots w'_{z_i}fw'_{z_i+p+1}\cdots w'_{q}-f'$)
\end{itemize}
\item[(Q)] $g:=\RED(h;S^{(i-1)}_n)$ is well-defined and $g\ne 0$. Moreover, we have $\{\exists!Z\}=\{C\in X_n^{\ast}\mid C\not\in\IRR(T_n)\}$ and $a_Z=\pm1$
where $g=\sum_{C\in X_n^{\ast}}a_CC$.
\item[(R)] Let $S^{(i)}_n$ be $S^{(i-1)}_n$ added with a reduction $s=(Z,\mp(g\mp Z))$
\end{enumerate}


\begin{Rem}\label{magicr}
We found the sequence based on a heuristic to establish (Q),
but we will not explain the detail since it is not conceptually understandable and not needed (see also \S\ref{spup}).
Again, to remain mathematically rigorous, it is good to think 
it is just a Ramanujan-given sequence to which computer can verify (P),(Q).
\end{Rem}

\begin{Def}\label{finalRS}
Let $R_n=S^{(m)}_n$ for $n\ne 7,11,15,22$.
When $n=7,11,15$ (resp. $22$), let $R_n=(s^{(m)}_0,s^{(m)}_1,s^{(m)}_2,s^{(m)}_4,\cdots,s^{(m)}_{N^{(m)}})$ 
(resp. $(s^{(m)}_0,s^{(m)}_1,s^{(m)}_2,s^{(m)}_3,s^{(m)}_5,\cdots,s^{(m)}_{N^{(m)}})$).
\end{Def}

The presentation so far uniquely define $R_n$ if we believe the phrases ``computer can verify'' in Remark \ref{notet} and Remark \ref{magicr} (and we will just do so from now on).
Modification of $S^{(m)}_{n}$ to define $R_n$ in Definition \ref{finalRS} just comes from the following fact on ``initial values'' $s^{(0)}_i=(W_i,f_i)$ where $f_i=\sum_{C}a^{(i)}_C C$.
\begin{itemize}
\item when $n\ne 7,11,15,22$, for all $1\leq i\leq N^{(0)}$ we have $C\in\AVO(T_n)$ if $a^{(i)}_C\ne 0$
\item when $n=7,11,15$, for all $1\leq i\leq N^{(0)},i\ne 3$ we have $C\in\AVO(T_n)$ if $a^{(i)}_C\ne 0$
\item when $n=22$, for all $1\leq i\leq N^{(0)},i\ne 4$ we have $C\in\AVO(T_n)$ if $a^{(i)}_C\ne 0$
\end{itemize}

Since $\forall s=(W,f)\in R_n,f\in\cor[\IRR(T_n)]$ by design, Proposition \ref{actprop} implies Proposition \ref{crucialprop}.
Recall ``well-defined'' in Proposition \ref{actprop} (resp. (Q)) means that even 
though the good order $\geq$ (see Proposition \ref{existsRS}) is not compatible with $R_n$ (resp. $S^{(i-1)}_n$) in general, loop in
algorithm \ref{algred} halts after finitely many steps.

\begin{Prop}\label{actprop}
For $w\in \AVO(T_n)$ and for $x\in X_n$, $\RED(wx;R_n)$ is well-defined.
\end{Prop}

\section{Computer calculation}
We verify Proposition \ref{actprop} by computer. 
To verify the case $n$, press 
\begin{quotation}
\verb|./test_bmr |$n$\verb| t|$n$\verb| s|$n$
\end{quotation}
We provide the computation log for $n$ by the name \verb|log|$n$.
All necessary ingredients are packed into 
\verb|http://www.kurims.kyoto-u.ac.jp/~tshun/BMR.tar.gz|.

\subsection{$T_n$}
As in Remark \ref{notet}, we just provide $T_n$ computed elsewhere by the text file \verb|t|$n$. 
For example, in a unix shell a command \verb|cat t4| provides
{\tiny{
\begin{verbatim}
6
3 0 0 0
3 1 1 1
3 1 0 1
5 0 1 0 0 1
5 1 0 0 1 1
5 1 1 0 0 1
\end{verbatim}
}}
\noindent{which means $T_4=\{sss,ttt,tst,stsst,tsstt,ttsst\}$.}
Note that in principle $T_n$ is computable by your own based on Definition \ref{condef}.
$|T_n|$ is given as follows.
\begin{center}
\begin{tabular}{c|ccccccccccccccccccc}
$n$     & 4 & 5  & 6 & 7  & 8  & 9  & 10 & 11  & 12 & 13 & 14 & 15  & 16 & 17 & 18  & 19  & 20 & 21 & 22 \\ \hline
$|T_n|$ & 6 & 13 & 6 & 62 & 14 & 13 & 43 & 185 & 24 & 33 & 11 & 106 & 44 & 49 & 138 & 551 & 36 & 30 & 66
\end{tabular}
\end{center}

\subsection{The sequences}
We also provide the sequence in Definition \ref{mseq} by the text file \verb|seq|$n$.
As in Remark \ref{magicr}, it plays a crucial role, but there is no easy way to understand and duplicate them.
The reason why we write down all explicitly in \S\ref{dat} is to ensure the self-containedness of the paper meaning that
in principle explanations of this paper with data in \S\ref{dat} can reproduce a computer proof of Proposition \ref{actprop}.

\subsection{C++ codes}
The list of codes and the synopsis is as follows.
\begin{enumerate}
\item \verb|test_bmr.cpp| that checks $|\AVO(T_n)|=|G_n|$, execute (P),(Q),(R) in \S\ref{mainex} and verifies Proposition \ref{actprop}
\item \verb|smartpointer.hpp| is an implementation of a smart pointer of shared type with a reference count (see ~\cite[\S14.2]{AK})
\item \verb|simple_polynomial.hpp| is an implementation of $\MA=\mathbb{Z}[a_1,\cdots,a_{\ell}]$ using \verb|tinycm.hpp|, \verb|smartpointer.hpp| and STL map container.
\item \verb|tinyncm.hpp| is an implementation of $\{s,t,u\}^{\ast}$ (in real, $0=s,1=t,2=u$)
\item \verb|tinycm.hpp| is an implementation of commutative monomials of $\mathbb{Z}[a_1,\cdots,a_{\ell}]$
\item \verb|crg.hpp| provides data for defining complex reflection groups of rank 2
\item \verb|magicnum.hpp| provides magic numbers for overflow detection (see \S\ref{overbf})
\end{enumerate}

While each code is short and based on just C++98 standard,
knowing the following correspondences in \verb|test_bmr.cpp| may improve readability.
\begin{itemize}
\item algorithm \ref{algred} is implemented as the sentence \verb|while ( is_reduceable2()>=0 ) {}| 
\item algorithm \ref{algenum} is implemented as the function \verb|quotient_dimension3| 
\end{itemize}

\subsection{Complication and overflow detection}\label{overbf}
To build a binary, you just press
\begin{quotation}
\verb|g++ test_bmr.cpp -o test_bmr -DBMR_OVERFLOW -O3|
\end{quotation}

The compilation option \verb|-DBMR_OVERFLOW| forces the followings check in a run.
\begin{itemize}
\item never involves NC monomial of length more than or equal to 28
\item never involves NC polynomial of length more than or equal to 200000
\item $|c|<100000$ for any coefficient $c$ of $\MA$ occurring in a computation
\item $\forall i,m_i<100$ for any monomial $a_1^{m_1}\cdots a_{\ell}^{m_{\ell}}$ of $\MA$ occurring in a computation
\item a reference count in an instance of \verb|smartpointer.hpp| never exceeds 100000
\end{itemize}

If you believe a pathological overflow never happens (and it really does not occur), you may remove \verb|-DBMR_OVERFLOW| in compilation.

\subsection{An example} When $n=4$, \verb|./test_bmr 4 t4 s4| displays
{\tiny{
\begin{verbatim}
dim=24
0th rule : <000>|-->(|<00>,[1*[1,0,]]|,|<0>,[1*[0,1,]]|,|<>,[1*[0,0,]]|)
1th rule : <111>|-->(|<11>,[1*[1,0,]]|,|<1>,[1*[0,1,]]|,|<>,[1*[0,0,]]|)
2th rule : <101>|-->(|<010>,[1*[0,0,]]|)
i=0
i=1
i=2
i=3
i=4
0th rule : <000>|-->(|<00>,[1*[1,0,]]|,|<0>,[1*[0,1,]]|,|<>,[1*[0,0,]]|)
1th rule : <111>|-->(|<11>,[1*[1,0,]]|,|<1>,[1*[0,1,]]|,|<>,[1*[0,0,]]|)
2th rule : <101>|-->(|<010>,[1*[0,0,]]|)
3th rule : <01001>|-->(|<10010>,[1*[0,0,]]|)
4th rule : <110010>|-->(|<10010>,[1*[1,0,]]|,|<0010>,[1*[0,1,]]|,|<011>,[1*[0,0,]]|)
5th rule : <10011>|-->(|<00110>,[1*[0,0,]]|,|<1001>,[1*[1,0,]]|,|<0110>,[-1*[1,0,]]|,|<110>,[-1*[0,1,]]|,|<100>,[1*[0,1,]]|)
6th rule : <11001>|-->(|<01100>,[1*[0,0,]]|,|<1001>,[1*[1,0,]]|,|<0110>,[-1*[1,0,]]|,|<011>,[-1*[0,1,]]|,|<001>,[1*[0,1,]]|)
0-th basis=().0 results in : (<0>,[1*[0,0,]])
0-th basis=().1 results in : (<1>,[1*[0,0,]])
1-th basis=(0).0 results in : (<00>,[1*[0,0,]])
1-th basis=(0).1 results in : (<01>,[1*[0,0,]])
2-th basis=(0,0).0 results in : (<00>,[1*[1,0,]])(<0>,[1*[0,1,]])(<>,[1*[0,0,]])
2-th basis=(0,0).1 results in : (<001>,[1*[0,0,]])
[SNIP]
21-th basis=(1,1).0 results in : (<110>,[1*[0,0,]])
21-th basis=(1,1).1 results in : (<11>,[1*[1,0,]])(<1>,[1*[0,1,]])(<>,[1*[0,0,]])
22-th basis=(1,1,0).0 results in : (<1100>,[1*[0,0,]])
22-th basis=(1,1,0).1 results in : (<0100>,[1*[0,0,]])
23-th basis=(1,1,0,0).0 results in : (<1100>,[1*[1,0,]])(<110>,[1*[0,1,]])(<11>,[1*[0,0,]])
23-th basis=(1,1,0,0).1 results in : (<01100>,[1*[0,0,]])(<1001>,[1*[1,0,]])(<0110>,[-1*[1,0,]])(<011>,[-1*[0,1,]])(<001>,[1*[0,1,]])
\end{verbatim}
}}
\label{outputbmr4}

We snipped the place of \verb|[SNIP]| to save space. This output can be read:
\begin{enumerate}
\item
$|\AVO(T_4)|=24$ and moreover computer enumerates as
\begin{align*}
\AVO(T_4) = \{1,s,ss,\cdots,tt,tts,ttss\}
\end{align*}
\item Starting from the defining relation (a),(b),(c),
new reductions (d),(e),(f),(g) are added according to the guidance of \verb|s4| (see also (P),(Q),(R) in \S\ref{mainex}).
\begin{align*}
\begin{array}{clc|ccl}
(a) & (sss,a_1ss+a_2s+1) & & & (d) & (stsst,tssts) \\
(b) & (ttt,a_1tt+a_2t+1) & & & (e) & (ttssts,a_1tssts+a_2ssts+stt) \\
(c) & (tst,sts)         & & & (f) & (tsstt,sstts+a_1tsst-a_1stts-a_2tts+a_2tss) \\
{} & {}                 & & & (g) & (ttsst,sttss+a_1tsst-a_1stts-a_2stt+a_2sst) 
\end{array}
\end{align*}
\item Proposition \ref{actprop} is verifies as follows (note it is snipped to save space).
\begin{align*}
\begin{array}{lc|cl}
1\cdot s = s             & {} & {} & tt\cdot s = tts                \\
1\cdot t = t             & {} & {} & tt\cdot t = a_1tt+a_2t+1       \\
s\cdot s = ss            & {} & {} & tts\cdot s = ttss              \\
s\cdot t = st            & {} & {} & tts\cdot t = stss              \\
ss\cdot s = a_1ss+a_2s+1 & {} & {} & ttss\cdot s = a_1ttss+a_2tts+tt \\
ss\cdot t = sst          & {} & {} & ttss\cdot t = sttss + a_1tsst - a_1stts - a_2stt + a_2sst
\end{array}
\end{align*}
\end{enumerate}

\subsection{On an implementation of $X_n^{\ast}$}\label{spup}
We first implement an element of $\{s,t,u\}^{\ast}$ as a 64bit unsigned integer unlike \verb|tinyncm.hpp|.
This ``\verb|bitncm.hpp|'' version is much faster and plays a role in finding the sequence in \S\ref{dat}
and convince us that the strategy in the paper operates well to verify the BMR conjecture.
But we do not provide \verb|bitncm.hpp| concerning portability and readability. 

\subsection{On an implementation of $\MA$}
It seems that an implementation of an element of $\MA$ via \verb|smartpointer.hpp|
(i.e.,adapting shallow copy) is 
a key technique stuff to accomplish the computation.
A heuristic reason
why a deep copy of an element of $\MA$
makes our computation unbearably slow
is that
we call \verb|is_reduceable2| (an application of a reduction) many (e.g. 49494726 when $n=19$) time and
in each call many parts of a NC polynomial to which a reduction applies are just copied.

\subsection{Resources}
On a server with Xeon E5-2643v3 (3.4GHz) and DDR4-2133, 
the run for $n=17,18,19$ took less than $1,8,105$ hours respectively (in a flash for other $n$) and 1GB memory is enough.
By design \verb|test_bmr.cpp| requires a single core. 

While there is a room for further speed up as in \S\ref{spup},
we judge that the elapsed time is acceptable for the purpose of verifying the BMR freeness and stop here.

\section{Appendix : The sequences}\label{dat}
\setcounter{subsection}{3}
\setlength{\columnseprule}{0.1pt}
{\tiny{
\begin{multicols}{8}
\subsection{4}
\begin{verbatim}
2 2 1 0
2 3 2 0
2 4 1 0
4 0 1 0
\end{verbatim}

\subsection{14}
\begin{verbatim}
2 2 2 0
2 3 3 0
1 3 1 0
4 0 1 0
5 2 3 0
0 7 1 0
7 6 4 0
3 4 3 0
6 1 1 0
1 11 1 0
5 0 1 0
13 10 6 0
9 7 5 0
8 0 1 0
\end{verbatim}

\subsection{6}
\begin{verbatim}
2 2 4 0
1 3 1 0
0 4 1 0
5 0 1 0
6 1 1 0
1 7 1 0
\end{verbatim}

\subsection{70}
\begin{verbatim}
3 4 2 0
4 0 1 0
4 3 1 0
4 6 2 0
5 3 2 0
1 6 1 0
3 6 1 0
11 7 1 1
3 10 1 0
4 5 1 0
3 14 2 0
4 7 2 0
5 2 1 0
2 17 1 0
5 4 1 0
10 1 1 0
0 20 1 0
10 17 2 0
11 17 2 0
14 14 2 0
4 24 2 0
17 4 1 0
18 0 1 0
27 21 3 0
8 27 3 0
19 1 1 0
21 9 3 0
22 0 1 0
22 24 4 0
0 31 1 0
0 32 1 0
21 35 3 0
0 36 1 0
1 8 1 0
1 24 1 0
2 34 1 0
3 22 1 0
4 15 1 0
5 17 1 0
8 18 3 0
11 44 3 0
9 8 3 0
3 33 2 0
10 34 2 0
0 48 1 0
14 34 2 0
35 50 3 0
50 35 3 0
10 43 2 0
16 26 5 0
17 44 2 0
18 50 2 0
27 51 3 0
38 17 2 0
45 49 4 0
46 17 2 0
54 0 1 0
61 50 3 0
49 61 3 0
1 62 1 0
3 45 1 0
3 62 1 0
8 25 3 0
10 67 3 0
10 66 2 0
0 69 1 0
29 53 4 0
11 71 3 0
50 22 2 0
34 64 4 0
\end{verbatim}

\subsection{17}
\begin{verbatim}
2 2 1 0
2 3 2 0
2 4 2 0
5 2 2 0
2 5 1 0
2 6 1 0
5 0 1 0
2 9 1 0
1 8 1 0
8 9 3 0
3 7 3 0
9 8 4 0
3 13 2 0
2 11 1 0
4 9 1 0
7 0 2 0
10 2 2 0
\end{verbatim}

\subsection{21}
\begin{verbatim}
2 2 4 0
1 3 1 0
2 3 5 0
0 4 1 0
1 5 1 0
0 7 1 0
6 0 1 0
9 1 1 0
1 10 1 0
0 11 1 0
6 7 5 0
5 9 4 0
8 0 1 0
15 1 1 0
1 16 1 0
16 15 8 0
13 0 2 0
19 1 1 0
1 20 1 0
8 7 5 0
10 9 4 0
\end{verbatim}

\subsection{48}
\begin{verbatim}
2 2 2 0
2 3 3 0
1 3 1 0
2 4 3 0
1 4 1 0
5 2 3 0
7 2 3 0
2 5 1 0
3 3 1 0
0 8 1 0
5 0 1 0
6 0 1 0
9 14 5 0
14 1 1 0
1 15 2 0
11 13 6 0
15 8 5 0
1 16 1 0
8 18 4 0
16 14 5 0
9 22 5 0
1 23 2 0
18 1 1 0
23 9 6 0
1 21 2 0
7 0 1 0
4 11 4 0
4 12 4 0
0 9 1 0
6 4 4 0
28 30 7 0
30 28 7 0
14 19 5 0
16 22 5 0
29 13 7 0
21 8 5 0
30 8 3 0
34 1 1 0
2 27 3 0
41 8 5 0
8 41 4 0
1 42 1 0
1 43 2 0
19 15 6 0
1 40 1 0
1 46 2 0
39 28 7 0
40 34 8 0
\end{verbatim}

\subsection{217}
\begin{verbatim}
3 4 2 0
4 0 1 0
4 3 1 0
4 6 2 0
5 3 2 0
3 6 1 0
10 7 1 1
4 5 1 0
3 12 2 0
4 7 2 0
5 2 1 0
2 15 1 0
5 4 1 0
9 3 2 0
10 15 2 0
13 12 3 0
15 4 1 0
16 0 1 0
8 22 3 0
1 6 1 0
3 24 1 0
5 15 1 0
8 16 3 0
10 27 3 0
7 27 3 0
9 4 1 0
9 8 3 0
10 26 2 0
14 3 2 0
14 21 5 0
15 12 1 0
15 27 2 0
17 1 1 0
23 22 3 0
24 15 2 0
26 4 1 0
12 12 2 0
3 41 2 0
31 15 2 0
31 27 3 0
34 0 1 0
6 45 2 0
39 0 1 0
0 47 1 0
47 22 3 0
1 8 1 0
2 19 1 0
3 39 1 0
4 29 2 0
5 26 1 0
10 19 2 0
10 28 2 0
15 26 1 0
3 57 2 0
18 8 3 0
24 1 1 0
0 60 1 0
22 61 3 0
61 18 4 0
61 48 4 0
0 63 1 0
65 57 0 1
0 64 1 0
24 26 2 0
25 3 2 0
25 21 5 0
26 27 2 0
30 1 1 0
31 26 2 0
33 3 2 0
74 18 4 0
74 48 4 0
2 75 1 0
45 74 3 0
4 78 2 0
8 75 3 0
8 76 3 0
11 45 2 0
15 58 2 0
15 75 2 0
27 37 4 0
33 40 6 0
2 65 1 0
34 77 4 0
39 77 4 0
3 89 2 0
40 7 4 0
0 91 1 0
48 74 3 0
50 15 2 0
52 10 3 0
53 84 5 0
59 15 2 0
67 22 3 0
67 79 6 0
31 58 3 0
86 0 1 0
88 52 6 0
36 31 4 0
89 56 0 1
91 77 5 0
75 105 5 0
101 74 3 0
2 99 1 0
3 66 1 0
3 68 1 0
10 32 2 0
0 111 1 0
112 0 1 0
76 113 5 0
112 101 5 0
0 115 1 0
10 43 2 0
10 44 2 0
10 87 3 0
10 102 2 0
12 39 1 0
17 39 1 0
17 43 2 0
17 44 2 0
23 75 3 0
23 76 3 0
24 54 2 0
22 93 3 0
24 80 3 0
26 58 2 0
26 129 4 0
27 77 2 0
2 89 1 0
31 19 2 0
31 28 2 0
38 127 5 0
10 136 3 0
76 112 5 0
31 136 3 0
45 96 3 0
47 75 3 0
27 141 4 0
47 76 3 0
50 26 2 0
58 72 5 0
59 26 2 0
66 113 4 0
68 96 5 0
3 148 2 0
69 40 6 0
71 31 4 0
0 151 1 0
77 120 5 0
87 97 7 0
96 30 3 0
10 155 3 0
96 151 6 0
98 79 6 0
104 105 6 0
115 113 5 0
129 19 5 0
141 124 0 1
148 118 0 1
151 96 6 0
76 164 5 0
154 0 1 0
2 148 1 0
3 100 1 0
3 104 1 0
3 159 2 0
10 73 2 0
0 171 1 0
172 0 1 0
172 101 5 0
0 174 1 0
15 155 2 0
17 66 1 0
17 71 2 0
17 94 1 0
17 97 2 0
24 43 2 0
24 125 3 0
27 96 2 0
27 147 4 0
30 43 2 0
31 43 2 0
31 44 2 0
49 76 3 0
100 173 5 0
114 75 3 0
114 76 3 0
129 89 5 0
131 45 3 0
141 45 3 0
174 173 6 0
182 43 6 0
3 162 1 0
10 146 2 0
10 156 2 0
10 167 3 0
17 100 1 0
24 73 2 0
24 97 2 0
24 142 3 0
30 73 2 0
30 97 2 0
66 164 4 0
68 140 5 0
76 171 4 0
10 209 3 0
165 0 1 0
31 209 3 0
115 164 5 0
151 140 6 0
167 186 9 0
175 149 8 0
3 184 1 0
10 177 2 0
10 186 2 0
215 0 1 0
24 183 3 0
\end{verbatim}

\subsection{26}
\begin{verbatim}
2 4 1 0
3 0 1 0
1 6 1 0
7 2 1 0
2 8 1 0
0 9 1 0
3 4 2 0
1 11 1 0
5 1 1 0
13 0 1 0
13 11 3 0
0 14 1 0
16 2 1 0
0 15 1 0
1 17 1 0
7 14 3 0
7 15 3 0
9 12 3 0
9 16 2 0
11 16 2 0
12 1 1 0
9 25 3 0
2 20 1 0
2 22 1 0
11 13 1 0
28 29 4 0
\end{verbatim}

\subsection{43}
\begin{verbatim}
2 4 1 0
4 3 0 1
1 6 1 0
5 1 1 0
8 0 1 0
0 9 1 0
10 2 1 0
1 11 1 0
3 4 2 0
0 13 1 0
7 1 1 0
15 0 1 0
2 16 1 0
17 2 1 0
0 18 1 0
1 19 1 0
7 4 3 0
15 13 4 0
9 17 3 0
10 11 2 0
12 12 1 0
1 25 1 0
8 26 3 0
11 25 3 0
14 1 1 0
19 25 3 0
21 25 3 0
22 15 4 0
14 21 3 0
3 29 1 0
29 31 4 0
1 35 1 0
10 36 3 0
1 37 1 0
10 17 1 0
1 39 1 0
13 14 2 0
0 41 1 0
16 42 5 0
43 2 1 0
1 44 1 0
8 10 1 0
0 45 1 0
\end{verbatim}

\subsection{15}
\begin{verbatim}
2 2 6 0
1 3 1 0
0 4 1 0
5 0 1 0
6 1 1 0
1 7 1 0
0 8 1 0
5 4 7 0
9 0 1 0
11 1 1 0
1 12 1 0
7 6 6 0
10 0 1 0
15 1 1 0
1 16 1 0
\end{verbatim}

\subsection{134}
\begin{verbatim}
4 2 1 0
4 3 0 1
0 3 1 0
1 5 1 0
2 8 1 0
3 6 4 0
9 1 1 0
4 11 1 0
6 4 1 0
7 1 1 0
14 0 1 0
1 15 1 0
7 4 2 0
8 8 1 0
2 10 1 0
3 6 2 0
4 19 1 0
8 17 3 0
14 5 3 0
16 2 1 0
19 22 5 0
1 25 1 0
4 25 2 0
13 1 1 0
15 4 1 0
29 22 5 0
1 30 1 0
17 1 1 0
19 15 4 0
1 33 1 0
4 33 2 0
22 30 5 0
25 16 3 0
26 1 1 0
38 0 1 0
26 4 2 0
2 40 1 0
4 40 1 0
1 42 1 0
20 15 4 0
27 1 1 0
45 0 1 0
28 5 3 0
29 15 4 0
31 1 1 0
32 5 3 0
43 1 1 0
49 34 6 0
1 52 1 0
1 48 1 0
1 50 1 0
38 8 2 0
2 23 1 0
1 57 1 0
4 57 2 0
4 34 2 0
58 1 1 0
30 61 5 0
59 1 1 0
63 0 1 0
1 64 1 0
65 8 3 0
61 8 2 0
2 55 1 0
0 68 1 0
1 69 1 0
45 8 2 0
63 8 2 0
33 72 6 0
4 70 2 0
8 40 1 0
15 16 2 0
15 26 4 0
25 14 2 0
25 65 3 0
1 71 1 0
27 36 5 0
30 58 5 0
40 32 3 0
55 83 6 0
66 45 5 0
79 71 1 1
39 45 4 0
40 43 4 0
2 88 1 0
89 88 1 1
4 83 2 0
39 85 5 0
46 92 6 0
1 46 1 0
41 1 1 0
4 95 1 0
83 96 6 0
96 49 6 0
2 72 1 0
16 49 3 0
53 0 1 0
4 96 1 0
4 98 1 0
72 31 5 0
77 16 3 0
99 1 1 0
0 81 1 0
2 56 1 0
2 105 1 0
72 52 5 0
4 99 1 0
16 22 1 0
28 112 5 0
4 38 2 0
61 0 1 0
2 113 1 0
4 58 2 0
117 8 3 0
89 11 2 0
96 8 3 0
98 0 1 0
120 45 5 0
0 121 1 0
123 18 5 0
4 114 1 0
26 36 5 0
35 84 6 0
55 96 5 0
1 128 1 0
62 29 5 0
106 18 4 0
125 67 7 0
0 104 1 0
2 102 2 0
4 91 1 0
4 135 1 0
15 114 3 0
72 14 2 0
\end{verbatim}

\subsection{57}
\begin{verbatim}
2 2 1 0
3 2 2 0
4 2 2 0
5 2 2 0
2 6 2 0
2 7 2 0
5 3 2 0
3 9 4 0
0 6 1 0
11 2 1 0
2 12 2 0
12 4 4 0
13 3 4 0
2 15 2 0
3 10 4 0
11 17 4 0
18 5 5 0
0 7 1 0
20 2 1 0
2 21 2 0
20 17 4 0
22 3 4 0
0 9 2 0
25 4 4 0
2 16 2 0
11 27 4 0
27 20 5 0
20 27 4 0
27 1 1 0
4 14 4 0
13 11 4 0
0 8 1 0
0 10 2 0
0 14 2 0
2 26 2 0
2 30 1 0
3 19 3 0
17 11 4 0
35 39 6 0
2 19 2 0
40 23 8 0
22 11 4 0
0 19 1 0
16 13 5 0
0 44 1 0
2 44 2 0
19 27 1 0
32 25 6 0
39 18 7 0
25 50 5 0
39 41 7 0
3 52 2 0
54 45 9 0
45 54 8 0
18 51 6 0
22 35 4 0
2 57 1 0
\end{verbatim}

\subsection{92}
\begin{verbatim}
2 2 4 0
3 2 5 0
4 2 5 0
3 1 1 0
4 1 1 0
6 0 1 0
5 1 1 0
9 0 1 0
7 0 1 0
0 8 1 0
7 11 5 0
7 12 4 0
1 12 1 0
15 1 1 0
16 0 1 0
12 5 5 0
0 11 1 0
1 19 1 0
9 12 4 0
20 1 1 0
22 0 1 0
6 15 2 0
7 10 5 0
12 15 5 0
0 10 1 0
1 27 1 0
28 1 1 0
27 28 10 0
6 16 2 0
19 5 5 0
0 14 2 0
1 33 1 0
30 28 10 0
34 1 1 0
9 11 5 0
7 22 2 0
36 0 1 0
12 20 5 0
9 10 5 0
9 18 5 0
0 32 1 0
1 43 1 0
44 1 1 0
19 15 5 0
6 20 2 0
45 0 1 0
7 20 2 0
19 20 5 0
33 4 4 0
9 21 6 0
46 20 9 0
0 51 1 0
1 54 1 0
55 1 1 0
16 17 6 0
26 5 5 0
0 24 1 0
50 20 9 0
0 52 2 0
1 61 1 0
62 1 1 0
53 5 5 0
47 5 5 0
7 13 5 0
6 44 2 0
36 46 8 0
9 25 5 0
31 15 6 0
63 0 1 0
26 15 5 0
47 15 5 0
7 49 5 0
0 65 1 0
0 49 1 0
67 4 4 0
0 64 1 0
1 78 1 0
79 1 1 0
16 48 6 0
16 50 5 0
38 15 6 0
9 75 4 0
3 58 2 0
53 15 5 0
9 78 4 0
0 77 1 0
0 81 2 0
16 68 6 0
74 5 5 0
0 84 2 0
72 44 8 0
0 90 2 0
\end{verbatim}

\subsection{164}
\begin{verbatim}
2 2 2 0
2 3 3 0
4 2 3 0
4 1 1 0
3 6 3 0
5 2 3 0
0 3 1 0
5 1 1 0
3 10 3 0
7 1 1 0
7 2 3 0
8 1 1 0
9 2 3 0
3 8 3 0
6 0 1 0
11 2 3 0
12 16 6 0
9 1 1 0
11 1 1 0
15 2 3 0
0 16 1 0
15 1 1 0
17 8 6 0
10 0 1 0
21 16 6 0
6 25 3 0
19 13 8 0
22 2 3 0
1 30 1 0
30 14 6 0
20 16 6 0
1 33 1 0
30 31 6 0
1 35 1 0
35 14 6 0
31 1 1 0
33 10 5 0
24 25 7 0
1 40 1 0
30 36 6 0
42 14 6 0
40 10 5 0
2 23 1 0
6 45 5 0
45 10 5 0
4 16 2 0
14 0 1 0
6 48 3 0
10 25 3 0
22 1 1 0
48 52 9 0
16 3 1 0
20 18 6 0
21 18 6 0
29 2 3 0
35 36 6 0
37 35 6 0
52 48 8 0
1 60 1 0
57 6 4 0
3 40 2 0
6 63 5 0
32 30 5 0
6 65 5 0
63 10 5 0
10 48 3 0
10 59 6 0
20 27 6 0
6 70 5 0
39 33 8 0
6 72 5 0
25 3 1 0
28 10 5 0
29 1 1 0
41 1 1 0
43 42 7 0
54 1 2 0
55 2 3 0
62 57 11 0
70 6 4 0
77 35 6 0
24 28 7 0
28 5 5 0
1 80 1 0
32 40 5 0
51 10 5 0
55 1 1 0
60 61 10 0
1 90 1 0
61 1 1 0
84 10 5 0
89 85 12 0
6 94 5 0
3 70 2 0
6 96 5 0
47 46 10 0
6 98 5 0
82 71 13 0
96 6 4 0
6 51 3 0
10 87 6 0
24 25 3 0
6 104 5 0
44 30 5 0
6 106 5 0
24 50 7 0
32 60 5 0
39 80 8 0
52 51 8 0
60 91 10 0
101 97 14 0
104 10 5 0
6 68 3 0
6 110 5 0
10 51 3 0
10 109 6 0
24 97 7 0
37 83 8 0
39 89 8 0
47 73 10 0
52 68 8 0
93 84 12 0
95 6 4 0
6 121 5 0
6 122 5 0
6 124 5 0
10 68 3 0
10 75 3 0
24 95 7 0
24 99 7 0
44 40 5 0
44 111 9 0
111 41 9 0
135 10 5 0
114 105 14 0
121 74 12 0
3 135 2 0
10 139 6 0
139 10 5 0
6 134 5 0
6 137 5 0
10 133 6 0
24 102 7 0
24 117 8 0
1 146 1 0
67 103 11 0
123 41 9 0
6 148 5 0
136 135 14 0
10 117 4 0
37 123 6 0
47 116 10 0
93 132 12 0
141 140 16 0
24 127 7 0
44 145 9 0
52 129 8 0
114 128 14 0
146 147 16 0
147 1 1 0
24 143 7 0
32 146 5 0
67 144 11 0
52 152 8 0
\end{verbatim}

\subsection{652}
\begin{verbatim}
3 4 1 0
5 4 2 0
4 3 2 0
3 7 2 0
5 0 1 0
4 9 2 0
4 5 2 0
8 4 2 0
8 7 3 0
8 0 1 0
9 6 2 0
10 7 3 0
2 6 1 0
10 4 2 0
1 4 1 0
6 17 2 0
4 17 1 0
2 20 1 0
9 20 2 0
17 0 1 0
13 4 2 0
13 0 1 0
12 0 1 0
21 27 4 0
13 7 3 0
11 18 3 0
18 3 2 0
6 22 2 0
27 6 2 0
2 15 1 0
30 0 1 0
35 21 3 0
4 22 1 0
37 2 2 0
7 21 1 0
6 28 2 0
4 30 1 0
41 14 3 0
9 15 2 0
9 32 2 0
19 4 2 0
16 4 1 0
19 0 1 0
25 0 1 0
37 48 5 0
29 0 1 0
27 20 2 0
29 4 2 0
39 7 2 0
35 37 3 0
34 0 1 0
37 2 1 0
47 6 2 0
48 6 2 0
56 0 1 0
59 21 3 0
57 0 1 0
1 21 1 0
4 38 1 0
20 7 1 0
63 6 3 0
3 57 1 0
61 21 3 0
66 9 3 0
0 67 1 0
7 41 1 0
8 37 2 0
9 33 2 0
7 37 1 0
9 23 2 0
3 70 1 0
9 61 1 0
11 61 1 0
20 28 2 0
14 61 1 0
21 22 2 0
27 15 2 0
24 61 1 0
27 32 2 0
35 36 3 0
59 37 3 0
48 20 2 0
47 18 3 0
61 37 3 0
53 4 2 0
53 7 3 0
90 4 1 0
14 91 3 0
26 91 3 0
91 61 4 0
91 4 1 0
6 49 2 0
47 91 3 0
39 91 3 0
63 10 3 0
89 0 1 0
52 0 1 0
65 13 4 0
28 92 5 0
2 92 1 0
31 99 5 0
63 11 3 0
14 102 3 0
70 14 3 0
9 92 2 0
100 21 3 0
50 91 3 0
101 6 2 0
46 7 1 0
82 6 2 0
0 88 1 0
64 65 5 0
114 0 1 0
26 102 3 0
105 65 5 0
4 95 1 0
0 119 1 0
9 51 2 0
14 70 1 0
6 95 2 0
7 70 1 0
11 57 1 0
10 65 2 0
9 58 2 0
9 117 1 0
35 60 3 0
27 23 2 0
20 49 2 0
32 97 5 0
55 133 5 0
4 133 1 0
4 135 1 0
3 135 2 0
35 90 3 0
27 136 4 0
9 71 2 0
0 140 1 0
141 0 1 0
24 142 3 0
2 142 1 0
48 136 4 0
2 135 1 0
135 4 1 0
35 80 3 0
135 76 0 1
2 139 1 0
133 4 1 0
0 151 1 0
105 106 5 0
0 153 1 0
94 154 7 0
2 152 1 0
2 154 1 0
68 102 6 0
3 158 2 0
27 33 2 0
26 103 3 0
59 36 3 0
133 100 5 0
48 32 2 0
101 20 2 0
35 54 3 0
48 15 2 0
61 36 3 0
55 140 4 0
73 91 3 0
66 17 2 0
171 27 4 0
59 80 3 0
77 4 2 0
68 70 6 0
100 37 3 0
98 4 1 0
64 106 5 0
59 90 3 0
14 177 3 0
49 93 6 0
9 93 2 0
61 80 3 0
82 136 4 0
100 90 3 0
127 105 6 0
28 180 5 0
146 112 8 0
28 109 5 0
177 4 1 0
135 134 5 0
188 0 1 0
174 0 1 0
101 136 4 0
35 98 3 0
126 120 6 0
193 21 3 0
133 134 5 0
2 109 1 0
152 59 4 0
7 177 1 0
7 151 2 0
18 158 5 0
202 15 5 0
32 202 5 0
9 72 2 0
177 193 6 0
9 81 2 0
172 177 7 0
9 86 2 0
2 158 1 0
27 58 2 0
26 177 3 0
27 51 2 0
32 184 5 0
27 139 2 0
14 78 2 0
27 44 2 0
24 169 1 0
3 209 2 0
9 180 2 0
20 196 4 0
48 23 2 0
35 110 3 0
61 54 3 0
26 181 3 0
57 37 2 0
47 103 3 0
17 204 3 0
59 60 3 0
48 33 2 0
61 60 3 0
59 54 3 0
61 98 3 0
3 234 2 0
193 98 6 0
0 236 1 0
9 83 2 0
55 237 5 0
0 238 1 0
240 0 1 0
240 100 6 0
97 2 1 0
134 2 1 0
66 22 2 0
101 32 2 0
50 103 3 0
39 177 3 0
100 80 3 0
35 85 3 0
132 0 1 0
28 182 5 0
2 234 1 0
157 4 2 0
49 161 6 0
156 167 9 0
154 59 4 0
254 0 1 0
170 258 7 0
35 170 3 0
205 4 1 0
193 37 3 0
59 98 3 0
256 4 2 0
264 48 5 0
196 43 7 0
94 213 7 0
227 201 7 0
158 35 3 0
105 116 5 0
258 21 3 0
241 134 5 0
201 172 7 0
0 261 1 0
28 274 5 0
9 160 2 0
200 154 7 0
7 235 2 0
32 278 5 0
18 209 5 0
278 33 6 0
9 248 3 0
205 59 4 0
14 255 3 0
234 59 4 0
257 154 7 0
27 81 2 0
27 72 2 0
3 277 2 0
35 162 3 0
48 51 2 0
9 131 2 0
27 86 2 0
59 110 3 0
27 83 2 0
26 78 2 0
61 85 3 0
28 221 5 0
59 85 3 0
163 133 5 0
17 281 3 0
35 130 3 0
61 170 3 0
3 303 2 0
9 164 2 0
0 305 1 0
306 0 1 0
48 58 2 0
306 100 6 0
258 170 7 0
0 310 1 0
55 311 5 0
24 312 3 0
61 148 3 0
2 209 1 0
49 226 6 0
50 181 3 0
35 176 3 0
55 238 4 0
319 48 5 0
2 303 1 0
26 187 3 0
59 170 3 0
149 264 6 0
253 231 10 0
193 36 3 0
248 208 10 0
61 110 3 0
268 81 8 0
0 329 1 0
258 37 3 0
9 320 3 0
237 59 4 0
196 158 7 0
28 277 5 0
172 187 7 0
7 261 2 0
78 337 8 0
202 152 5 0
325 0 1 0
337 81 8 0
274 35 3 0
27 128 2 0
7 304 2 0
3 336 2 0
9 214 2 0
283 205 8 0
222 35 3 0
344 58 7 0
9 212 2 0
14 316 3 0
15 154 1 0
27 131 2 0
27 160 2 0
20 180 1 0
24 319 1 0
35 230 3 0
279 59 4 0
27 161 2 0
27 122 2 0
307 134 5 0
48 86 2 0
28 284 5 0
48 72 2 0
24 329 2 0
61 130 3 0
17 349 3 0
48 112 2 0
50 189 3 0
55 306 3 0
59 130 3 0
370 48 5 0
61 173 3 0
9 321 3 0
94 322 7 0
35 263 3 0
117 319 4 0
61 176 3 0
204 155 5 0
193 195 6 0
0 380 1 0
35 233 3 0
193 54 3 0
228 2 1 0
239 2 1 0
193 60 3 0
242 241 6 0
193 98 3 0
258 36 3 0
370 2 1 0
315 287 11 0
200 213 7 0
193 327 6 0
296 282 11 0
257 213 7 0
391 4 2 0
259 264 6 0
396 27 4 0
9 398 3 0
263 59 4 0
0 388 1 0
349 304 9 0
3 392 2 0
59 176 3 0
339 167 9 0
337 275 9 0
311 312 8 0
390 0 1 0
340 241 6 0
379 59 4 0
381 59 4 0
333 133 5 0
9 291 2 0
338 35 3 0
193 338 6 0
9 288 2 0
18 381 3 0
18 336 5 0
15 237 1 0
35 290 3 0
9 293 2 0
0 421 1 0
27 223 2 0
422 0 1 0
417 72 8 0
26 316 3 0
27 296 3 0
163 237 5 0
340 319 7 0
48 182 2 0
27 214 2 0
48 160 2 0
16 424 3 0
433 86 8 0
28 351 5 0
47 255 3 0
59 230 3 0
61 224 3 0
24 370 1 0
309 307 7 0
50 255 3 0
28 392 5 0
61 250 3 0
117 370 4 0
61 230 3 0
49 359 6 0
172 298 7 0
149 396 6 0
59 327 3 0
6 424 1 0
234 37 1 0
35 323 3 0
227 177 3 0
237 37 1 0
258 54 3 0
204 239 5 0
35 299 3 0
193 260 6 0
0 458 1 0
193 85 3 0
26 298 3 0
258 60 3 0
281 155 5 0
258 263 7 0
0 464 1 0
59 224 3 0
59 263 3 0
401 59 4 0
313 2 1 0
312 2 1 0
408 306 7 0
9 353 2 0
193 394 6 0
9 354 2 0
259 396 6 0
27 291 2 0
283 286 8 0
410 381 10 0
15 213 1 0
16 465 3 0
27 292 2 0
394 35 3 0
480 160 9 0
18 459 3 0
484 128 9 0
24 393 2 0
23 237 1 0
14 355 2 0
411 381 10 0
9 362 2 0
27 288 2 0
48 214 2 0
163 311 5 0
424 312 8 0
26 446 4 0
48 296 3 0
59 323 3 0
49 426 6 0
50 335 3 0
61 318 3 0
28 478 5 0
100 249 3 0
61 299 3 0
61 294 3 0
98 283 4 0
149 397 6 0
48 226 2 0
180 473 8 0
200 322 7 0
193 130 3 0
193 162 3 0
285 465 8 0
281 239 5 0
259 311 5 0
59 294 3 0
258 205 3 0
193 347 6 0
0 517 1 0
35 371 3 0
258 323 7 0
0 520 1 0
258 85 3 0
449 4 2 0
523 27 4 0
400 264 6 0
453 287 11 0
259 397 6 0
483 239 8 0
505 523 11 0
59 299 3 0
463 239 7 0
465 239 7 0
333 237 5 0
417 379 8 0
468 401 11 0
410 523 11 0
333 465 8 0
15 381 1 0
16 521 3 0
539 212 10 0
18 518 3 0
27 354 2 0
541 288 11 0
27 353 2 0
3 512 2 0
242 409 7 0
47 446 4 0
24 429 1 0
49 495 6 0
28 512 5 0
48 284 2 0
61 357 3 0
59 467 4 0
50 446 4 0
528 59 4 0
193 230 3 0
193 233 3 0
258 130 3 0
285 475 9 0
285 396 6 0
193 327 3 0
432 531 13 0
16 562 3 0
32 563 5 0
563 214 10 0
528 4 1 0
562 4 1 0
0 567 1 0
193 467 7 0
400 465 8 0
0 569 1 0
55 571 5 0
35 437 3 0
117 429 4 0
2 568 1 0
410 518 10 0
534 364 12 0
400 475 9 0
358 428 9 0
411 518 10 0
513 239 7 0
540 239 8 0
562 59 4 0
425 155 5 0
59 357 3 0
400 237 5 0
193 376 6 0
0 587 1 0
3 566 1 0
400 396 6 0
15 322 1 0
564 59 4 0
50 442 3 0
589 354 12 0
16 586 4 0
595 276 11 0
521 239 7 0
9 575 3 0
61 466 3 0
61 457 3 0
258 230 3 0
285 586 9 0
333 588 8 0
421 90 2 0
393 375 9 0
410 529 11 0
417 516 8 0
320 422 6 0
476 581 14 0
468 489 11 0
507 2 1 0
15 518 1 0
575 542 15 0
613 4 2 0
555 614 13 0
15 562 1 0
571 572 12 0
59 519 3 0
594 567 13 0
48 582 2 0
55 620 4 0
278 463 6 0
242 521 6 0
613 0 1 0
622 432 13 0
592 564 13 0
14 566 1 0
32 392 1 0
50 501 3 0
48 426 2 0
117 527 4 0
163 570 6 0
242 527 7 0
285 590 9 0
258 401 3 0
542 598 15 0
547 2 1 0
55 637 4 0
563 528 10 0
555 562 12 0
20 627 4 0
278 584 6 0
358 533 9 0
3 633 2 0
565 155 5 0
7 635 1 0
4 644 1 0
15 588 1 0
32 478 1 0
32 426 1 0
50 549 3 0
149 578 6 0
163 578 6 0
596 267 7 0
32 495 1 0
7 654 1 0
\end{verbatim}

\subsection{45}
\begin{verbatim}
2 2 3 0
2 3 4 0
0 3 1 0
4 0 1 0
4 3 5 0
5 2 4 0
1 8 1 0
4 8 4 0
6 1 1 0
0 10 1 0
0 11 1 0
1 13 1 0
5 0 1 0
6 10 4 0
9 0 1 0
13 17 6 0
2 4 1 0
15 1 1 0
1 20 1 0
17 1 1 0
0 22 1 0
6 11 4 0
7 0 1 0
25 19 8 0
8 6 2 0
8 27 4 0
9 8 4 0
14 0 1 0
15 10 4 0
15 16 7 0
16 6 5 0
28 8 6 0
0 24 1 0
1 32 1 0
11 27 4 0
17 13 3 0
35 37 9 0
21 27 4 0
21 28 7 0
39 35 10 0
40 21 8 0
16 15 3 0
27 34 6 0
28 20 6 0
44 46 12 0
\end{verbatim}

\subsection{50}
\begin{verbatim}
2 2 8 0
1 3 1 0
0 4 1 0
5 0 1 0
6 1 1 0
1 7 1 0
0 8 1 0
5 4 9 0
9 0 1 0
11 1 1 0
1 12 1 0
0 13 1 0
10 0 1 0
15 1 1 0
1 16 1 0
0 17 1 0
10 4 9 0
14 0 1 0
20 1 1 0
1 21 1 0
12 11 9 0
18 0 1 0
24 1 1 0
1 25 1 0
7 6 6 0
9 8 7 0
19 0 1 0
29 1 1 0
1 30 1 0
27 31 13 0
9 17 7 0
12 23 9 0
16 15 8 0
18 8 7 0
28 0 1 0
37 1 1 0
1 38 1 0
7 24 6 0
23 26 9 0
25 6 6 0
18 17 7 0
36 17 12 0
38 37 14 0
7 27 6 0
28 8 7 0
25 24 6 0
25 40 11 0
44 0 1 0
50 1 1 0
1 51 1 0
\end{verbatim}

\subsection{77}
\begin{verbatim}
3 4 0 1
1 3 1 0
1 4 1 0
0 7 1 0
8 2 1 0
1 9 1 0
2 10 1 0
5 3 3 0
0 12 1 0
6 1 1 0
14 0 1 0
14 12 4 0
2 15 1 0
17 2 1 0
0 18 1 0
1 19 1 0
2 16 1 0
8 15 4 0
8 16 4 0
11 1 1 0
19 24 4 0
24 0 1 0
13 1 1 0
27 0 1 0
1 28 1 0
29 2 1 0
2 30 1 0
0 26 1 0
10 29 3 0
1 22 1 0
34 2 1 0
1 23 1 0
2 35 1 0
10 27 2 0
10 38 3 0
11 10 3 0
12 8 2 0
13 9 4 0
15 17 3 0
17 18 3 0
21 1 1 0
45 0 1 0
0 46 1 0
27 28 3 0
27 42 3 0
37 29 3 0
0 44 1 0
1 48 1 0
9 8 2 0
10 52 3 0
12 13 2 0
12 47 2 0
15 14 2 0
57 55 6 0
17 46 3 0
18 38 3 0
51 60 6 0
22 34 4 0
37 38 3 0
37 39 5 0
39 10 4 0
45 42 3 0
45 49 5 0
49 27 4 0
53 58 6 0
0 67 1 0
27 48 3 0
34 46 3 0
34 59 5 0
34 66 5 0
49 47 4 0
2 73 1 0
17 67 3 0
39 35 4 0
38 78 4 0
49 45 3 0
78 80 8 0
\end{verbatim}

\end{multicols}
}}

\end{document}